\documentclass{article}
\usepackage{amssymb,amsmath,amsthm,graphicx}
\usepackage[all,color]{xy}
\usepackage{caption}

\def\tr{\triangleright}

\newtheorem{theorem}{Theorem}

\newtheorem{corollary}[theorem]{Corollary}

\theoremstyle{definition}
\newtheorem{example}{Example}
\newtheorem{definition}{Definition}
\newtheorem{remark}{Remark}

\date{}

\title{\Large \textbf{Quandle Coloring Quivers of Surface-Links}}

\author{Jieon Kim\footnote{mail: jieonkim7@gmail.com. Supported by Young Researchers Program through the National Research Foundation of Korea (NRF), funded by the Ministry of Education, Science and Technology (NRF-2018R1C1B6007021).} \and
	Sam Nelson\footnote{Email: Sam.Nelson@cmc.edu. Partially supported by Simons Foundation collaboration grant 316709}\and
	Minju Seo\footnote{Email: mingzuu@pusan.ac.kr} \and
}

\begin{document}
	\maketitle
	
	\begin{abstract}
Quandle coloring quivers are directed graph-valued invariants of oriented knots and 
links, defined using a choice of finite quandle $X$ and set $S\subset\mathrm{Hom}(X,X)$
of endomorphisms. From a quandle coloring quiver, a polynomial knot invariant known as
the \textit{in-degree quiver polynomial} is defined. We consider quandle coloring quiver 
invariants for oriented surface-links, represented by marked graph diagrams. We provide 
example computations for all oriented surface-links with ch-index up to 10 for
choices of quandles and endomorphisms.
	\end{abstract}
	
	\parbox{5.5in} {\textsc{Keywords:} Quandle coloring quivers, surface-links, 
in-degree polynomial.
		
		\smallskip
		
		}

	\section{\large\textbf{Introduction}}\label{I}
	
	In \cite{Lo}, a diagrammatic approach to the study of knotted and linked surfaces
	in $\mathbb{R}^4$ using knotted 4-valent graphs with decorated vertices was 
        introduced. Every compact surface embedded in 
	$\mathbb{R}^4$ can be positioned by ambient isotopy to have all its maxima 
	with respect to an axis $w$ at $w=1$, all of its minima at $w=-1$ and all of 
	its saddle points at $w=0$; this is known as a \textit{hyperbolic splitting}.
	Then the cross-section of the surface with the $w=0$ hyperplane is a knotted
	4-valent graph with the properties that (1) each vertex represents a saddle 
	point and (2) smoothing the vertices by taking a slightly higher or lower
	cross-section in the $w$ direction results in an unlink. These diagrams, 
	known as \textit{marked graph diagrams} or \textit{marked vertex diagrams} or 
	\textit{ch-diagrams}, determine surface-links up to ambient isotopy. 
	The combinatorial moves on marked 
	graph diagrams encoding ambient isotopy in $\mathbb{R}^4$ are known as 
	\textit{Yoshikawa moves}; in \cite{KJL2}, the first listed author and 
        coauthors discuss the history of these moves and provide generating sets 
        of the moves.
	
	\textit{Quandles} are algebraic structures with axioms defined from the
	Reidemeister moves of classical knot knot theory. Given a finite quandle $X$,
	the set of quandle homomorphisms from the fundamental quandle of a classical
	knot to $X$ is an invariant, and its cardinality is known as the 
	\textit{quandle counting invariant}. See \cite{EN} and the references therein for more.
	
	In \cite{CN}, the second-listed author and a coauthor defined
	an enhancement of the quandle counting invariant for classical
	knots, the \textit{quandle coloring quiver}. This quiver-valued
	invariant of classical knots and links categorifies the quandle counting 
	invariant, and the extra information in the quiver can distinguish knots with
	the same counting invariant. Since dealing with large quivers directly can be
	problematic, further invariants were defined, notably the \textit{in-degree 
		polynomial}.
	
	Quandle-based invariants have been defined for surface-links and studied in
	papers such as \cite{KJL, KJL2}. In this paper we define the quandle coloring quiver
	for surface-links. The paper is organized as follows. In Section \ref{SSL} we
	review the basics of quandles and of surface-links. In Section \ref{QCQ} we
	define quandle coloring quivers for surface-links and provide examples and 
	computations. In particular we show that the new invariant is a proper
	enhancement in the sense that it can distinguish surface-links with the same
	quandle counting invariant. We further extend the in-degree polynomial to 
	the case of surface-links. We conclude in Section \ref{Q} with some questions 
	for future work.

	\section{\large\textbf{Quandles and Surface-Links}}\label{SSL}

	In this section, we review surface-links, marked graph diagrams and quandles.
        We begin with the basics of quandle theory. See \cite{EN} and the references therein for more detail.

	\begin{definition}
		A set $X$ equipped with a binary operation $\vartriangleright$ is a \textit{quandle} if it satisfies 
		\begin{enumerate}
			\item $x\vartriangleright x=x$ for all $x\in X$,
			\item for each $y\in X$, the map $f_y:X\to X$ defined by $f_y(x)=x\vartriangleright y$ is a bijection, and
			\item $(x\vartriangleright y)\vartriangleright z=(x\vartriangleright z)\vartriangleright(y\vartriangleright z)$ for all $x,y,z\in X$.
		\end{enumerate}
	\end{definition}

\begin{example}
Let $K$ be an oriented classical knot. The \textit{fundamental quandle} or 
\textit{knot quandle} of $K$, denoted $Q(K)$,
is the quandle generated by generators corresponding to arcs in a diagram
of $K$ with the relation
\[\begin{array}{c}\includegraphics{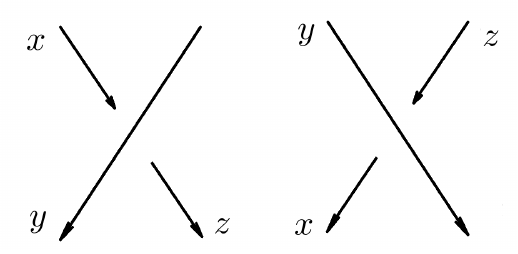}\\ z=x\tr y\end{array}\]
at each crossing. More formally, elements of the knot quandle are equivalence 
classes of quandle words in these generators modulo the equivalence relation
generated by the quandle axioms and the crossing relations. There is also
a geometric interpretation of the knot quandle in terms of path homotopy 
classes 
of paths from a basepoint to the boundary torus of the knot complement; see
\cite{EN} and the references therein for more detail.
\end{example}

\begin{example}
Let $G$ be a group. $G$ has several quandle structures including:
\begin{itemize}
\item \textbf{Core quandles}. Setting 
\[x\tr y= xy^{-1}x\] 
defines a
quandle structure on any group. In the case of a cyclic group $G=\mathbb{Z}_n$
written additively, this structure is sometimes called a \textit{cyclic quandle}
and has the operation 
\[x\tr y=2y-x.\]
The famous Fox colorings of knots are really homomorphisms from the knot quandle
to the core quandle of $\mathbb{Z}_3$.
\item \textbf{Conjugation quandles}. For any integer $n\in\mathbb{Z}$, setting
\[x\tr y=y^{-n}xy^n\]
defines a quandle structure on $G$. The case $n=1$ for $G$ the knot group of a knot
is closely related to the knot quandle, and the cases $n\ge 2$ are
closely related to the \textit{Wada groups} of a knot, the $n=2$ case of which was 
shown to be a complete invariant for classical knots in \cite{NN}. The $n=0$ case
(or indeed, any $n$ for $G$ abelian) is called the \textit{trivial quandle} on
$n$ elements.
\end{itemize}
\end{example}

\begin{example}
Let $X$ be any module over $\mathbb{Z}[t^{\pm 1}]$. Then $X$ is a quandle with
\[x\tr y=tx+(1-t)y\]
known as an \textit{Alexander quandle}. The \textit{Fundamental Alexander quandle}
of a classical knot, with a generator for each arc and the relation
\[\begin{array}{c}\includegraphics{jk-sn-ms-2.pdf}\\ z=tx+(1-t)y\end{array}\]
for each crossing, is the Alexander module of the knot considered as a quandle.
As such, it determines the Alexander polynomials of the knot.
\end{example}

\begin{example}
Let $V$ be a vector space over a field $\mathbb{F}$ and $[,]:V\times V\to \mathbb{F}$
a symplectic form. Then $V$ is a quandle, known as a \textit{symplectic quandle},
under the operation
\[\vec{x}\tr \vec{y}=\vec{x}+[\vec{x},\vec{y}]\vec{y}.\]
\end{example}

\begin{example}
Any quandle structure on a finite set $X=\{1,2,\dots,n\}$ can be expressed 
via an operation table. For example, $\mathbf{Core}(\mathbb{Z}_3)$ has the 
operation table
\[\begin{array}{l|lll}
\tr & 1 & 2 & 3\\ \hline
1 & 1 & 3 & 2 \\
2 & 3 & 2 & 1 \\
3 & 2 & 1 & 3.
\end{array}\]
\end{example}

	\begin{definition}
		Let $X,Y$ be quandles with multiplication operations indicated by $\vartriangleright_X$ and $\vartriangleright_Y$ respectively. A map $f:X\to Y$ is a \textit{quandle homomorphism} given that $f(a\vartriangleright_X b)=f(a)\vartriangleright_Y f(b)$ for any $a,b\in X$.
	\end{definition}
	
	\begin{definition}
		Let $L$ be an oriented knot or link and $X$ a finite quandle called the coloring quandle. The coloring space $\mathrm{Hom}(\mathcal{Q}(L),X)$ is the space of quandle homomorphisms from $\mathcal{Q}(L)\to X$. The \textit{quandle  counting  invariant} is the cardinality of the coloring space, $\left| Hom(\mathcal{Q}(L),X) \right|$, which we will denote by $\Phi^\mathbb{Z}_X (L)$.	
	\end{definition}

We now recall the basics of surface-links.

	\begin{definition}
		A \textit{surface-link} is a closed surface smoothly embedded in $\mathbb{R}^4$.
		If a surface-link is (orientable and) oriented, then we call it an \textit{oriented\ surface-link}.
	\end{definition}

	Two surface-links $F$ and $F'$ are said to be equivalent 
	if there exists an orientation-preserving homeomorphism $h : \mathbb{R}^4\to\mathbb{R}^4$ such that $h(F) = F'$.
	When $F$ and $F'$ are oriented, it is assumed that $h|_F : F\to F'$ is an orientation
	preserving homeomorphism. An equivalence class of a surface-link is called a surface-link type.
	A surface-knot is trivial(or unknotted) if it is obtained from some standard surfaces in $\mathbb{R}^4$ (i.e. a standard 2sphere, a standard torus, and standard projective planes in $\mathbb{R}^4$) by taking a connected sum.
	A surface-link is trivial(or unknotted) if it is obtained from some trivial surface-knots by taking a split union.

	\begin{definition}
		A \textit{ marked  vertex  graph} or simply a \textit{marked graph} is a spatial graph $G$ in $\mathbb{R}^3$ which satisfies that $G$ is a finite regular graph possibly with 4-valent vertices, say $v_1, v_2, ..., v_n$; each vertex $v_i$ is a rigid vertex (that is, we fix a rectangular neighborhood $N_i$ homeomorphic to $\{(x, y)| -1\le x,\ y\le 1\}$, where $v_i$ corresponds to the origin and the edges incident  to $v_i$ are represented by $x^2 = y^2$); each vertex $v_i$ has a \textit{marker} which is the interval on $N_i$ given by $\{(x, 0)|-{1\over 2}\le x\le {1\over 2} \}$.
	\end{definition}

	\begin{definition}
		An $\it orientation$ of a marked graph G is a choice of an orientation for each edge of G in such a way that every vertex in G looks like $\xy (0,3);(3,0)**@{-} ?<*\dir{<},
		(3,0);(6,-3) **@{-} ?>*\dir{>},
		(0,-3);(2,-1) **@{-} ?>*\dir{>}, 
		(2,-1);(3,0) **@{-}, 
		(3,0);(4,1) **@{-} , 
		(4,1);(6,3) **@{-} ?<*\dir{<}, 
		(1,0);(5,0)**@{-} ,
		(1,0.1);(5,0.1)**@{-},
		(1,-0.1);(5,-0.1)**@{-},
		(1,0.2);(5,0.2)**@{-},
		(1,-0.2);(5,-0.2)**@{-}, \endxy$ or $\xy (0,3);(2,1)**@{-} ?>*\dir{>},
		(2,1);(4,-1)**@{-},
		(4,-1);(6,-3) **@{-} ?<*\dir{<},
		(0,-3);(3,0) **@{-} ?<*\dir{<},  
		(3,0);(6,3) **@{-} ?>*\dir{>}, 
		(1,0);(5,0)**@{-} ,
		(1,0.1);(5,0.1)**@{-},
		(1,-0.1);(5,-0.1)**@{-},
		(1,0.2);(5,0.2)**@{-},
		(1,-0.2);(5,-0.2)**@{-}, \endxy$.
		A marked graph is said to be \textit{orientable} if it admits an orientation.
		By an \textit{oriented marked graph}, we mean an orientable marked graph with a fixed orientation.
	\end{definition}
	
	\begin{definition}
		If two oriented marked graphs are ambient isotopic in $\mathbb{R}^3$ with keeping rectangular neighborhoods, an orientation and markers, they are \textit{equivalent}.
	\end{definition}
	
	\begin{definition}
		For any surface-link $F$, there exists a surface-link $F'$ satisfying the following:
		\begin{enumerate}
			\item $F'$ is equivalent to $F$ and has only finitely many Morse's critical points,
			\item all maximal points of $F'$ lie in $\mathbb{R}^3_1$,
			\item all minimal points of $F'$ lie in $\mathbb{R}^3_{-1}$,
			\item all saddle points of $F'$ lie in $\mathbb{R}^3_0$.			
		\end{enumerate}
		We call a representation $F'$ in the previous theorem, a \textit{hyperbolic splitting} of $F$.
	\end{definition}
	
	\begin{definition}
		The zero section $\mathbb{R}^3_0 \cap F'$ of the hyperbolic splitting $F^{'}$ gives us a 4-valent graph. We assign to each vertex a marker that informs us about one of the two possible types of saddle points. Then we obtain a diagram representing $F$, which is called a \textit{marked graph diagram} of $F$.
		
		\[ \resizebox{0.65\textwidth}{!}{\includegraphics{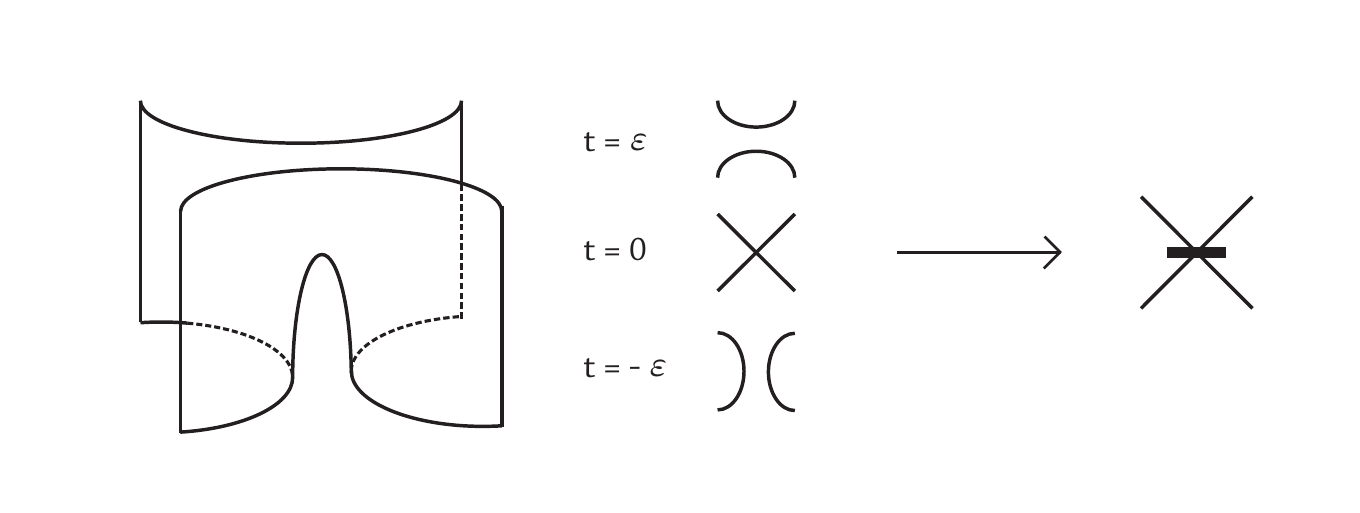}} \]
	       	%caption

	\end{definition}

	\begin{definition}
		An oriented marked graph $G$ in $\mathbb{R}^3$ can be described as usual by a diagram $D$ in $\mathbb{R}^2$, which is an oriented link diagram in $\mathbb{R}^2$ possibly with some marked 4-valent vertices whose incident four edges have orientations illustrated as above, and is called an \textit{oriented marked graph diagram} of $G$.
		
			\[ \includegraphics{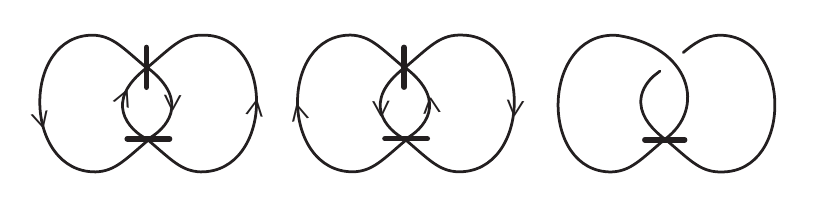} \]

	\end{definition}
	
	\begin{definition}
		A (oriented) marked graph diagram D is \textit{admissible} if both resolutions $L_{+}(D)$ and $L_{-}(D)$ are link diagrams of trivial links.
		
		\[\includegraphics[height=3cm, width=10cm]{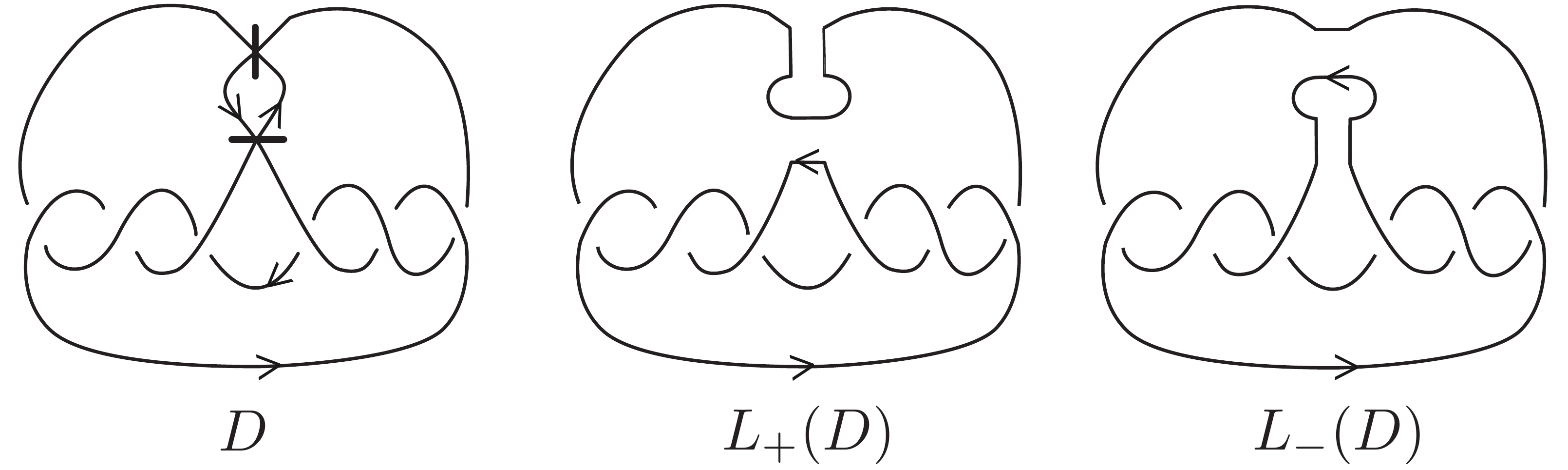} \]
		%caption
	\end{definition}

Given an admissible diagram, we can connect the two smoothings with a 
surface with saddle points where marked vertices were and cap off the 
unlinked components to obtain a broken surface diagram of the surface-link.

\begin{example}
The pictured marked graph diagram determines the pictured surface-link.
\[\includegraphics{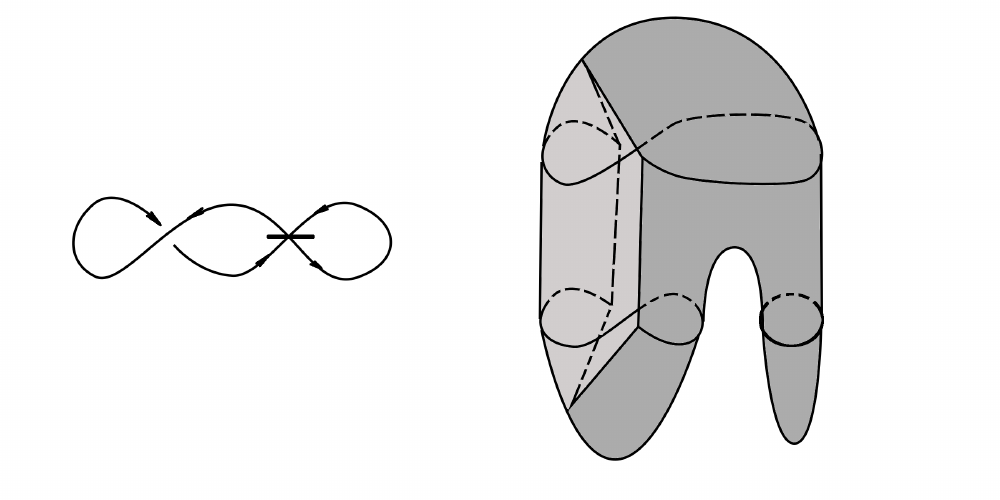}\]
\end{example}

	In summary, we have the following:
	\begin{theorem} (See \cite{KJL,KJL2,Lo} for more details)

		\begin{enumerate}
			\item For an admissible marked graph diagram $D$, there is a
			surface-link $L$ represented by $D$.
			\item Let $L$ be a surface-link. Then there is an admissible
			marked graph diagram $D$ such that $L$ is represented by $D$.
		\end{enumerate}
	\end{theorem}

\begin{remark}
We note that non-admissible marked vertex diagrams represent not
closed surface-links but cobordisms between the classical links given by
the upper and lower smoothings.
\end{remark}
	
The combinatorial moves on marked graph diagrams capturing ambient isotopy of 
surface-links are known as the \textit{Yoshikawa moves}. In addition to the
three classical Reidemeister moves, one generating set of oriented Yoshikawa 
moves is:
\[\includegraphics{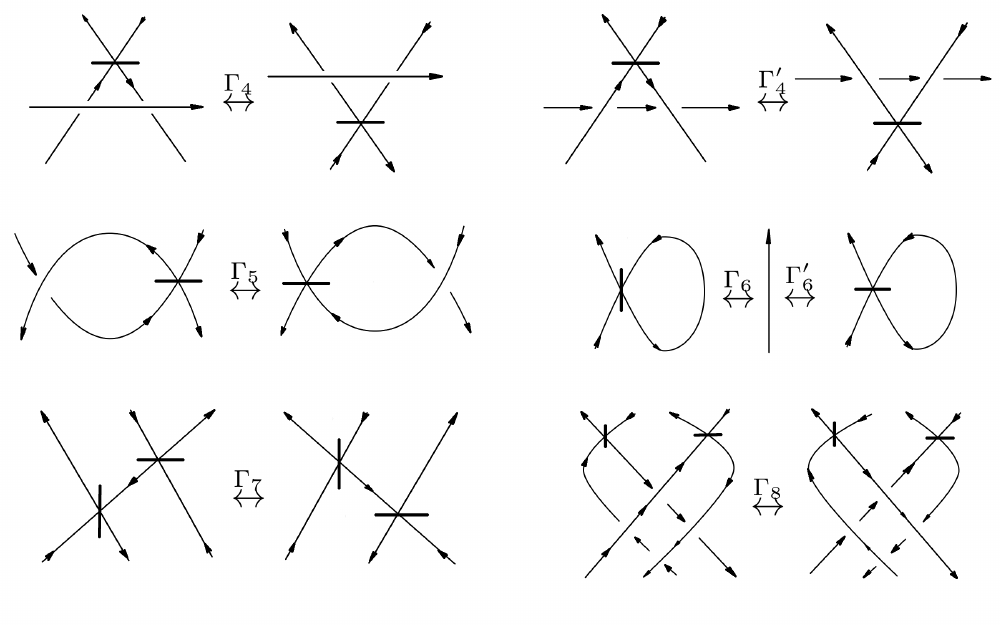}\]
See \cite{KJL2} for more.

	\section{\large\textbf{Quandle Coloring Quivers of Surface-Links}}\label{QCQ}

Next we recall some facts about quandle coloring quivers. See \cite{CN} for more.

	\begin{definition}
		Let $X$ be a finite quandle and $L$ an oriented link. For any set of quandle endomorphisms $S\subset Hom(X,X)$, the associated \textit{quandle coloring quiver}, denoted $Q^S_X (L)$, is the directed graph with a vertex for every element $f\in Hom(\mathcal{Q}(L),X)$ and an edge directed from $f$ to $g$ when $g = \phi f$ for an element $\phi \in S$. Important special cases include the case $S = Hom(X,X)$, which we call the \textit{full quandle coloring quiver} of $L$ with respect to X, denoted $Q_X (L)$, and the case when $S = \{\phi \}$ is a singleton, which we will denote by $Q^\phi_X (L)$.	
		
		\[	\includegraphics[height=4cm, width=5.5cm]{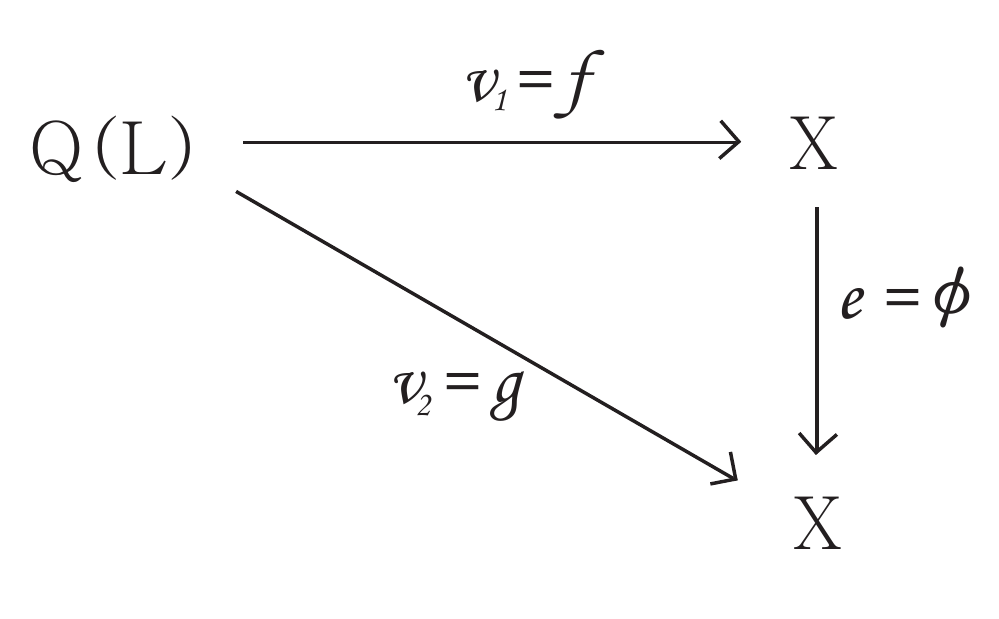}\]
		
	\end{definition}
	
In \cite{CN}, we find the following result:

	\begin{theorem}
		Let $X$ be a finite quandle, $S\subset Hom(X,X)$ and $L$ an oriented link. Then the quiver $Q^S_X (L)$ is an invariant of $L$.
	\end{theorem}
	
	\begin{corollary}
		Any invariant of directed graphs applied to $Q^S_X (L)$ defines an invariant of oriented links.
	\end{corollary}
	
	\begin{example}\label{ex1}
Let $X$ be the quandle with operation table
\[\begin{array}{r|rrrrrr}
\tr & 1 & 2 & 3 & 4 & 5 & 6 \\ \hline
1 & 1 & 3 & 2 & 1 & 1 & 1 \\
2 & 3 & 2 & 1 & 2 & 2 & 2 \\
3 & 2 & 1 & 3 & 3 & 3 & 3 \\
4 & 5 & 5 & 5 & 4 & 4 & 5 \\
5 & 4 & 4 & 4 & 5 & 5 & 4 \\
6 & 6 & 6 & 6 & 6 & 6 & 6.
\end{array}
\]
The endomorphism ring $\mathrm{Hom}(X,X)$ has 68 elements including for example 
$\phi=[6,6,6,5,4,2]$, i.e., the map sending $1$ to $\phi(1)=6$, $2$ to $\phi(2)=6$,
\dots, $\phi(6)=2$.
Then the trefoil knot $3_1$ has quandle coloring quiver $Q^{\phi}_X(3_1)$ given by
\[\includegraphics{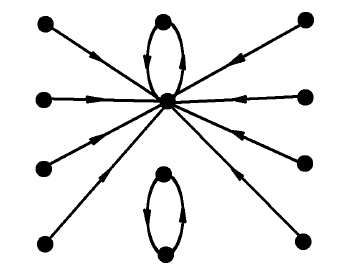}.\]

%		The Hopf link $L2a1$ has five colorings by the quandle $X$ specified by the operation table as shown below. 
%		\begin{center}
%			\begin{tabular}{l|llll}
%				$\vartriangleright$	& 1 & 2 & 3 &  \\ \hline%
%				1	& 1 & 1 & 2 &  \\
%				2	& 2 & 2 & 1 &  \\
%				3	& 3 & 3 & 3 & 
%			\end{tabular}
%		\end{center}		
%		Then the quandle homomorphism $\phi:X\to X$ defined by $\phi (1)=2$, $\phi (2)=1$, and $\phi (3)=3$ yields the following quandle coloring quiver $Q^\phi_X (L)$
%		
%		\[\includegraphics[height=5cm, width=9cm]{image/ex5-1}\]
		
	\end{example}
	
	\begin{example}\label{ex2}
The full quandle coloring quiver for the $(4,2)$-torus link with respect to the quandle given by the operation
table
		\begin{center}
			\begin{tabular}{l|llll}
				$\vartriangleright$	& 1 & 2 & 3 &  \\ \hline%
				1	& 1 & 1 & 2 &  \\
				2	& 2 & 2 & 1 &  \\
				3	& 3 & 3 & 3 & 
			\end{tabular}
		\end{center}
is
		\[	\includegraphics{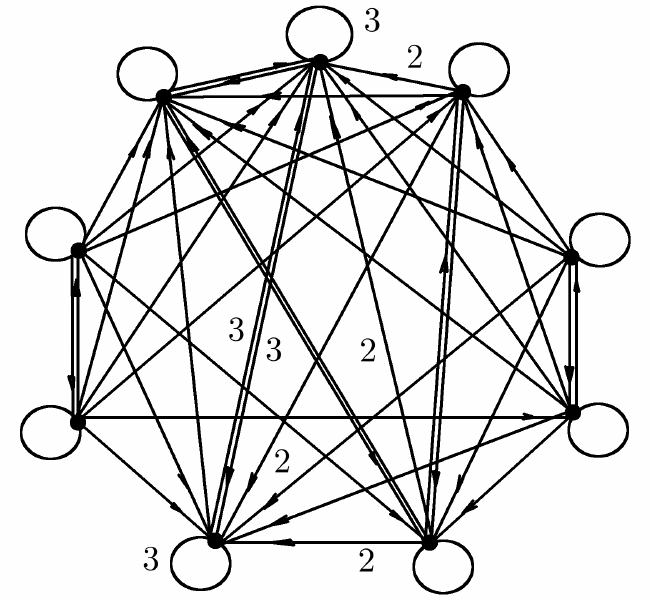}\]
where we indicate multiple edges with numbers.
	\end{example}
As noted in \cite{CN},
	the out-degree $deg^{-}(f)$ of every vertex is $\lvert S \rvert$; however, different vertices may have different in-degrees $deg^{+}(f)$. In-degrees of each vertex can be encoded as a polynomial knot invariant:
	
	\begin{definition}
		Let $X$ be a finite quandle, $S\subset Hom(X,X)$ a set of quandle endomorphisms, $L$ an oriented link and $Q^S_X (L)$ the associated quandle coloring quiver with set of vertices $V(Q^S_X (L))$. Then the \textit{in-degree quiver polynomial} of $L$ with respect to $X$ is\\
		\[\Phi^{deg^+}_{X,S} (L)=\sum_{f\in V(Q^S_X (L))} u^{deg^+ (f)}.\]
		If $S = \{\phi\}$ is a singleton we will write $V(Q^S_X (L))$ as $V(Q^\phi_X (L))$ and $\Phi^{\deg^+}_{X,S}(L)$ as $\Phi^{\deg^+}_{X,\phi}(L)$, and
and if $S = Hom(X,X)$ we will write $V(Q^S_X (L))$ as $V(Q_X (L))$
and $\Phi^{\deg^+}_{X,S}(L)$ as$\Phi^{\deg^+}_{X}(L)$.
	\end{definition}
	
	\begin{example}
		In Example \ref{ex2}, the $(4,2)$-torus link $L4a1$ has in-degree quiver polynomial $\Phi^{deg^+}_{X,\phi} (L)=5+u+2u^2+u^4$ with respect to the quandle $X$ and endomorphism $\phi=[1,1,2]$.
	\end{example}

We will now extend quandle coloring quivers to the case of surface-links.

	\begin{definition}
		Let X be a quandle and let $D$ be an oriented marked graph diagram. Let $A(D)$ be the set of arcs of $D$. A \textit{coloring of $D$ by $X$}, also called an \textit{$X$-coloring of $D$}, is an assignment of elements of $X$ to the elements of $A(D)$ as shown:
\[\includegraphics{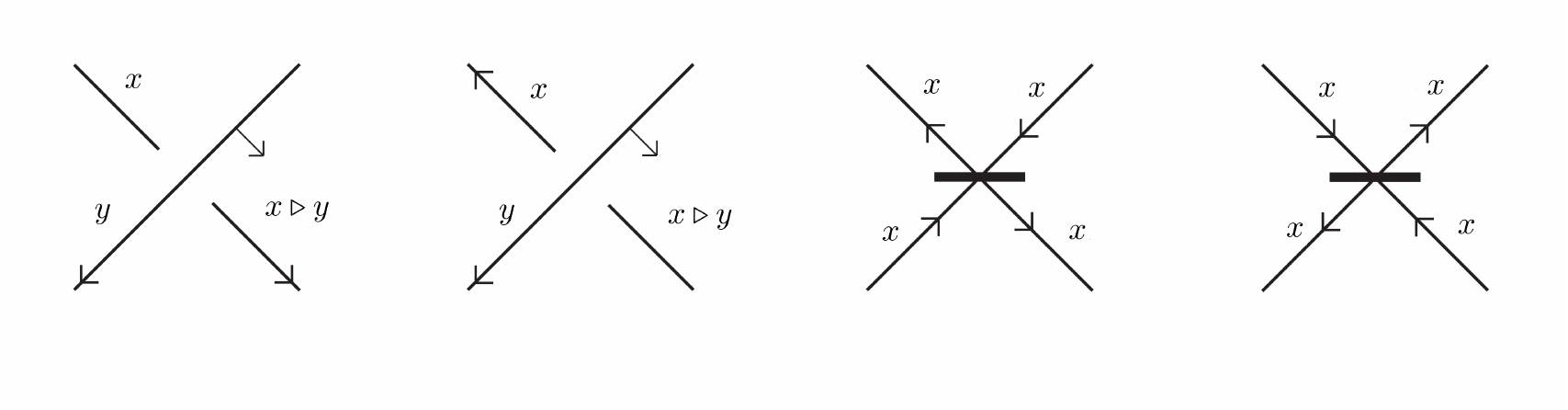}\]  
	\end{definition}
	
	We denote by $Col_X (D)$ the set of all X-colorings of D.

	\begin{theorem}
		Let $X$ be a finite quandle, $L$ an oriented surface-link and $D$ a marked graph diagram of $L$. Then,
		\[\sharp Col_X(L)=\sharp Col_X(D)\]
	\end{theorem}

We now have our main theorem:

	\begin{theorem}
		Let $X$ be a finite quandle, $S\subset Hom(X,X)$ and $L$ an oriented surface-link. Then the quiver $Q^S_X (L)$ is an invariant of $L$.
	\end{theorem}

\begin{proof}
As noted in \cite{KJL} and elsewhere, the set of quandle colorings is 
preserved by oriented Yoshikawa moves. Then as noted in \cite{CN}, the 
quandle coloring quiver is determined by the set of quandle homomorphisms
from the fundamental quandle of the surface-link to the coloring quandle.
\end{proof}

\begin{example}
Let $X$ be the quandle given by the operation table
\[\begin{array}{r|rrrr}
\tr & 1 & 2 & 3 & 4 \\ \hline
1 & 1 & 1 & 4 & 3 \\
2 & 2 & 2 & 2 & 2 \\
3 & 4 & 3 & 3 & 1 \\
4 & 3 & 4 & 1 & 4
\end{array}\]
and let $\phi=[2,4,2,2]$ be the quandle endomorphism mapping $1,3,4\in X$ to 
$2\in X$ and mapping $2\in X$ to $4\in X$. Then the surface-links $6^{0,1}_1$
and $8_1$ both have ten $X$-colorings, but are distinguished by their quandle
colorings quivers as shown:
\[\includegraphics{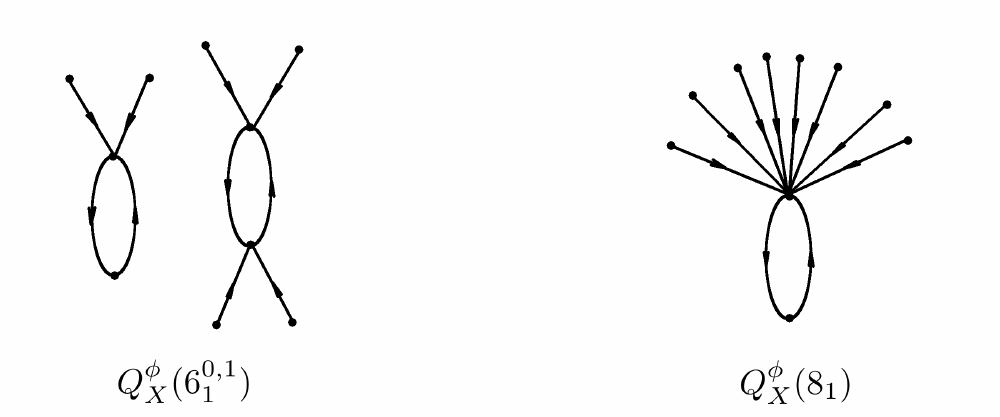}\]
In particular, this example shows that the quandle coloring quiver of a 
surface link is not determined by the number of quandle colorings of the
surface-link and hence is a proper enhancement.
\end{example}

\begin{corollary}
		Let $X$ be a finite quandle, $S\subset Hom(X,X)$ and $L$ an oriented surface-link. Then the in-degree polynomial $\Phi^{\mathrm{deg}_+}_X (L)$ is an invariant of $L$.
\end{corollary}

	\begin{example}
		Let $X$ be the dihedral  $3$-quandle, $Y$ the dihedral  $4$-quandle and $Z$ the tetrahedral quandle given by the operation tables as shown below.
		\[\begin{array}{l|lll}
		\tr & 1 & 2 & 3\\ \hline
		1 & 1 & 3 & 2 \\
		2 & 3 & 2 & 1 \\
		3 & 2 & 1 & 3, \\
		\end{array}\qquad
		\begin{array}{l|llll}
		\tr & 1 & 2 & 3 & 4\\ \hline
		1 & 1 & 3 & 1 & 3 \\
		2 & 4 & 2 & 4 & 2 \\
		3 & 3 & 1 & 3 & 1 \\
		4 & 2 & 4 & 2 & 4,
		\end{array}\qquad
		\begin{array}{l|llll}
		\tr & 1 & 2 & 3 & 4\\ \hline
		1 & 1 & 4 & 2 & 3 \\
		2 & 3 & 2 & 4 & 1 \\
		3 & 4 & 1 & 3 & 2 \\
		4 & 2 & 3 & 1 & 4.
		\end{array}\]
		
		 Let $L$ be an oriented surface-link with ch-index $\chi(L)\le 10$ presented by marked graph diagrams as shown below.
		
			\[\includegraphics[height=8cm, width=12cm]{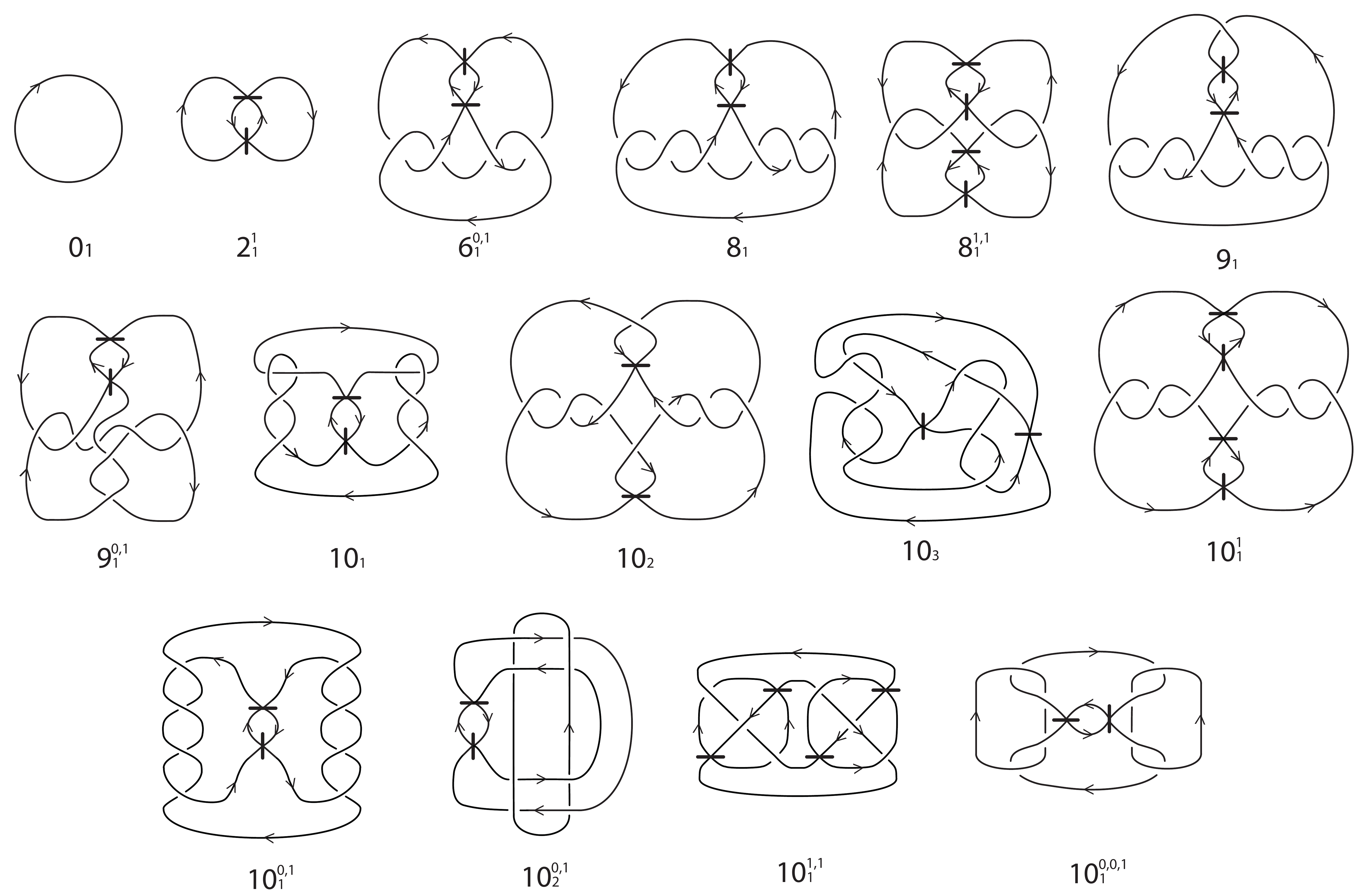}\]
		
		 Then for each $L$, we have in-degree quiver polynomials $\Phi^{deg^+}_X (L)$, $\Phi^{deg^+}_Y (L)$ and $\Phi^{deg^+}_Z (L)$ as follows:
		
		\begin{center}
			\begin{tabular} {|l|l|l|l|}
				\hline
				$L$ & $\Phi^{deg^+}_X (L)$ & $\Phi^{deg^+}_Y (L)$ & $\Phi^{deg^+}_Z (L)$\\
				\hline
				$0_1$ & $3u^9$ & $4u^{16}$ & $4u^{16}$\\
				\hline
				$2^2 _1$ & $3u^9$ & $4u^{16}$ & $4u^{16}$\\
				\hline
				$6^{0,1} _1$ & $3u^9$ & $4u^8 +4u^{24}$ & $4u^{16}$\\
				\hline
				$8_1$ & $6u^6 +3u^{12}$ & $4u^{16}$ & $12u^{12} +4u^{28}$\\
				\hline
				$8^{1,1} _1$ & $3u^9$ & $4u^8 +4u^{24}$ & $4u^{16}$\\
				\hline
				$9_1$ & $6u^6 +3u^{12}$ & $4u^{16}$  & $4u^{16}$\\
				\hline
				$9^{0,1} _1$ & $3u^9$ & $8u^8 +4u^{16} +4u^{32}$  & $4u^{16}$\\
				\hline
				$10_1$ & $3u^9$ & $4u^{16}$ & $12u^{12} +4u^{28}$\\
				\hline
				$10_2$ & $6u^6 +3u^{12}$ & $4u^{16}$ & $4u^{16}$\\
				\hline
				$10_3$ & $3u^9$ & $4u^{16}$ & $4u^{16}$\\
				\hline
				$10^1 _1$ & $6u^6 +3u^{12}$ & $4u^{16}$ & $12u^{12} +4u^{28}$\\
				\hline
				$10^{0,1} _1$ & $3u^9$ & $8u^8 +4u^{16} +4u^{32}$ & $12u^{12} +4u^{28}$\\
				\hline
				$10^{0,1} _2$ & $6u^6 +3u^{12}$ & $8u^8 +4u^{16} +4u^{32}$ & $4u^{16}$\\
				\hline
				$10^{1,1} _1$ & $3u^9$ & $4u^8 +4u^{24}$ &$4u^{16}$\\
				\hline
				$10^{0,0,1} _1$ & $6u^6 +3u^{12}$ & $24u^8 +4u^{24} +4u^{52}$ & $12u^{12} +4u^{28}$\\
				\hline
			\end{tabular}
		\end{center}
	\end{example}

We conclude this section with an observation:

\begin{remark}\label{rem:subq}
We observe that a coloring of a marked graph diagram is also a coloring 
of the top and bottom smoothed diagrams $T$ and $B$. It follows that for any 
cobordism $L$ from $T$ to $B$, the quandle coloring quiver of $L$ is a subquiver
of the quandle coloring quivers of $T$ and $B$ respectively.
\end{remark}
	
\section{\large\textbf{Questions}}\label{Q}

In this paper, we have only initiated the study of quandle coloring quiver 
invariants of surface-links via marked graph diagrams. In this section we 
collect some open questions for future work in this area.

\begin{itemize}
\item Can the observation in Remark \ref{rem:subq} be applied to find obstructions
to knot concordance or cobordisms?
\item For which quandles is the quandle coloring quiver of cobordism precisely the
intersection of the quivers of the top and bottom smoothed diagrams?
\item What other properties of quandle-colored marked graph diagrams
can be used to enhance the quandle coloring quiver?
\item Since a quiver is a category, the quandle coloring quiver is a categorification
of the quandle coloring invariant of oriented surface-links. What categorical properties
do these categories have for particular classes of quandles, for a choice of surface-link?
\end{itemize}

	\bibliography{jk-sn-ms}{}
	\bibliographystyle{abbrv}

	\bigskip

	\noindent
	\textsc{Department of Mathematics, \\
		Pusan National University, \\
		Busan 46241, Republic of Korea
	}

	\medskip
	
	\noindent
	\textsc{Department of Mathematical Sciences \\
		Claremont McKenna College \\
		850 Columbia Ave. \\
		Claremont, CA 91711}

\end{document}